\documentclass[12pt]{article}

\usepackage{geometry}
\usepackage{amsthm}
\usepackage{amssymb}
\usepackage{amsmath}
\usepackage[normalem]{ulem}
\usepackage{graphicx}
\usepackage[colorlinks=true,
linkcolor=webgreen,
filecolor=webbrown,
citecolor=webgreen]{hyperref}

\usepackage{color}
\usepackage{xcolor}

\definecolor{webgreen}{rgb}{0,.5,0}
\definecolor{webbrown}{rgb}{.6,0,0}
\definecolor{red}{rgb}{1,0,0}

\theoremstyle{plain}
\newtheorem{theorem}{Theorem}[section]
\newtheorem{corollary}[theorem]{Corollary}
\newtheorem{lemma}[theorem]{Lemma}

\theoremstyle{definition}

\newtheorem{example}[theorem]{Example}

\theoremstyle{remark}
\newtheorem{remark}[theorem]{Remark}

\usepackage{booktabs}

\usepackage{float}
\usepackage[pageref]{backref}
\renewcommand*{\backref}[1]{}
\renewcommand*{\backrefalt}[4]{%
\ifcase #1 (Not cited.)%
\or (Cited on page~#2.)%
\else (Cited on pages~#2.)%
\fi}

\allowdisplaybreaks

\numberwithin{equation}{section}

\begin{document}

\title{General convolution sums involving Fibonacci $m$-step numbers}

\date{}
\author{Robert Frontczak${}^{1}$ and Karol Gryszka${}^{2}$}
\maketitle
\begin{center}
\footnotesize
\textsc{
${}^{1}$Independent Researcher \\ Reutlingen, Germany \\
${}^{2}$Corresponding author \\ Institute of Mathematics \\ University of the National Education Commission, Krakow \\
    Podchor\c{a}\.{z}ych 2, 30-084 Krak{\'o}w, Poland}\par\nopagebreak
  \textit{E-mail addresses}: \texttt{robert.frontczak@web.de} \\ \texttt{karol.gryszka@uken.krakow.pl}
\end{center}

\begin{abstract}
In this paper, using a generating function approach, we derive several new convolution
sum identities involving Fibonacci $m$-step numbers. As special instances of the results
derived herein, we will get many new and known results involving Fibonacci, Tribonacci, Tetranacci and Pentanacci numbers.
In addition, we establish some general results providing insights into the inner structure of such convolutions.
Finally, some mixed convolutions involving Fibonacci $m$-step numbers and Jacobsthal and Pell numbers will be stated.
\end{abstract}


% \subjclass[2020]{Primary 11A99, 33B10; Secondary 11B37, 11A67}

% \textcolor{magenta}{TODO:}
\section{Motivation and Preliminaries}

The motivation for writing this paper comes from three recently published papers by Dresden and his collaborators:
Dresden and Wang \cite{Dresden1,Dresden3} and Dresden and Tulskikh \cite{Dresden2}.
In these papers, convolutions involving important number sequences like Fibonacci numbers $F_n$, Lucas numbers $L_n$,
Pell numbers $\mathcal{P}_n$, Jaconsthal numbers $J_n$, Tribonacci numbers $T_n$, and others are studied.
Dresden and Wang close their article \cite{Dresden3} with a short discussion of the two seemingly unrelated convolutions
\begin{equation}\label{eq:conv_sum_JF}
\sum_{j=0}^n J_j F_{n-j} = J_{n+1} - F_{n+1} \quad (\mbox{see}\,\,\cite{Koshy})
\end{equation}
and
\begin{equation}\label{F-T-conv}
\sum_{j=0}^n T_j F_{n-j} = T_{n+2} - F_{n+2} \quad (\mbox{see}\,\,\cite{BenBook,Frontczak2018,Frontczak2019}).
\end{equation}
A third example of this nature is the following convolution
\begin{equation}\label{P-F-conv}
\sum_{j=0}^n \mathcal{P}_j F_{n-j} = \mathcal{P}_{n} - F_{n} \quad (\mbox{see}\,\,\cite{Greubel,Seiffert}).
\end{equation}
A hidden link between \eqref{F-T-conv} and \eqref{P-F-conv} will be revealed later. At this point we can clearly see 
that identities \eqref{eq:conv_sum_JF}--\eqref{P-F-conv} look suspiciously similar.
Dresden and Wang \cite{Dresden3} ask "... if there are other general convolution formulas waiting to be discovered."?
The answer to this questions is "Definitely Yes!". Building on a generating function approach
(see also \cite{Frontczak2019,Komatsu2018,Komatsu2019,KomatsuLi2019,Prodinger2019})
for Fibonacci $m$-step numbers we will prove many new convolution identities, recovering known results as special cases, 
including the identities presented above.

We note that the generating function approach is not a novel discovery. However, the referenced articles show
that this approach has received a lot of attention in the recent years. In this context, we mention the latest article by
Gessel and Kar \cite{GesselKar} who provide an extensive analysis of (binomial) convolutions of sequences having rational generating functions.

As for the Fibonacci $m$-step sequences, there seems to be a lack of a general approach to some of the convolution identities involving these sequences and other recurrence relations, mentioned at the beginning of the article and many more. Our response to this is a step forward towards a better understanding of the general inner structure of these convolutions. To keep things coherent and focused on a single topic,
this article is devoted exclusively to Fibonacci $m$-step numbers. Thus, identities related to Lucas ($m$-step) sequences
(see for example \cite{Dresden1, Dresden2}) will be the subject of another study.

We start with a definition. The $m$-step Fibonacci numbers are defined for all $m\geq 1$ by 
\begin{equation*}
F_{1}^{(m)} = 1,\quad \text{and}\quad F_n^{(m)} = 0 \quad\mbox{for all}\quad n=-(m-2),\ldots,0
\end{equation*}
($F_0^{(1)}=0$) and for all $n\geq 2$,
\begin{equation*}
F_n^{(m)} = \sum_{j=1}^{m} F_{n-j}^{(m)}.
\end{equation*}
Their arithmetic structure was studied in a recent article by the second author \cite{Gryszka2022b}.
The ordinary generating function for $m$-step Fibonacci numbers is given by
\begin{equation*}
F^{(m)}(x) = \frac{x}{1-x-x^2-\cdots - x^m}.
\end{equation*}

We introduce the following notation for a selection of particular Fibonacci $m$-step numbers.
Let $F_n^{(2)}=F_n$, $F_n^{(3)}=T_n$, $F_n^{(4)}=Q_n$, and $F_n^{(5)}=P_n$ denote the Fibonacci, Tribonacci, Tetranacci, and Pentanacci numbers, respectively. So, the following can be stated.
\begin{align*}
    F_n&=F_{n-1}+F_{n-2}, & &F_0=0, F_1=1,\\
    T_n&=T_{n-1}+T_{n-2}+T_{n-3}, & & T_0=0, T_1=T_2=1,\\
    Q_n&=Q_{n-1}+Q_{n-2}+Q_{n-3}+Q_{n-4}, & & Q_0=0, Q_1=Q_2=1, Q_3=2,\\
    P_n&=P_{n-1}+P_{n-2}+P_{n-3}+P_{n-4}+P_{n-5}, & &P_0=0, P_1=P_2=1, P_3=2, P_4=4,%\\
\end{align*}
Let us denote by $F(x)$, $T(x)$, $Q(x)$, and $P(x)$ the ordinary generating functions of these sequences, respectively. Hence, we have
$$F(x)=F^{(2)}(x),\quad T(x)=F^{(3)}(x), \quad Q(x)=F^{(4)}(x)\quad \text{and}\quad P(x)=F^{(5)}(x).$$
% \begin{align*}
%     F(x) = \frac{x}{1-x-x^2} &= \sum\limits_{n=0}^{\infty}F_n x^n, \\
%     T(x) = \frac{x}{1-x-x^2-x^3} &= \sum\limits_{n=0}^{\infty} T_n x^n, \\
%     Q(x) = \frac{x}{1-x-x^2-x^3-x^4} &= \sum\limits_{n=0}^{\infty}Q_n x^n, \\
%     P(x) = \frac{x}{1-x-x^2-x^3-x^4-x^5} &= \sum\limits_{n=0}^{\infty} P_n x^n.
% \end{align*}

These special candidates will be used later to highlight particular cases of the results obtain in this paper.

\section{General convolution identities with Fibonacci $m$-step numbers}

% This section is for general formulas for Fibonacci $m$-step numbers. We consider two kinds of formulas:
% convolutions for different values of $m$ and convolutions for the same $m$.
This section is for general formulas for Fibonacci $m$-step numbers. We consider several forms of formulas.
First, we investigate convolutions of $F_n^{(m)}$ with $F_n^{(m')}$ for $m\neq m'$.
Then we show a few somewhat curious identities with convolutions exhibiting what we call a ''switch effect''.
One such convolution is
\begin{equation*}
\sum_{j=0}^{n} F_{j}\big (T_{n-j} - Q_{n-j}\big ) = \sum_{j=0}^{n-1} \big ( F_{j} - T_{j}\big ) Q_{n-1-j},
\end{equation*}
where ''switching'' is referred to ''switch the parentheses''. Finally, we consider convolutions those steps differ by $2$.

\subsection{Mixed convolution}

\begin{theorem}\label{th:general_conv_m_step}
For all $n\geq m$ we have
\begin{equation}\label{eq:simplest_conv_sum_in_the_article}
    \sum_{j=0}^{n-m} F_j^{(m)} F_{n-m-j}^{(m+1)}= F_n^{(m+1)} - F_n^{(m)}.
\end{equation}
\end{theorem}
\begin{proof}
Notice that
\begin{equation*}
    1-x-x^2-\cdots-x^{m}=\frac{x}{F^{(m)}(x)}.
\end{equation*}
Subtracting $x^{m+1}$ and rearranging we obtain
\begin{equation*}
    1-x-x^2-\cdots - x^{m} - x^{m+1} = \frac{x-x^{m+1}F^{(m)}(x)}{F^{(m)}(x)},
\end{equation*}
and thus
\begin{equation*}
    \frac{F^{(m)}(x)}{x-x^{m+1}F^{(m)}(x)} = \frac{F^{(m+1)}(x)}{x}.
\end{equation*}
This gives
\begin{equation}\label{eq:F^m+1-F^m-x^mF^m+1F^m}
    F^{(m+1)}(x) - F^{(m)} (x) = x^m F^{(m)}(x)F^{(m+1)}(x)
\end{equation}
or
\begin{align*}
   \sum_{n=0}^\infty (F_n^{(m+1)} - F_n^{(m)}) x^n
		&= x^m \Big (\sum_{n=0}^\infty F_n^{(m)}x^n\Big )\Big (\sum_{n=0}^\infty F_n^{(m+1)}x^n\Big ) \\
    &= \sum_{n=0}^\infty \sum_{j=0}^n F_j^{(m)} F_{n-j}^{(m+1)} x^{n+m} \\
    &= \sum_{n=m}^\infty \sum_{j=0}^{n-m} F_j^{(m)} F_{n-m-j}^{(m+1)} x^{n}.
\end{align*}
This proves the theorem as for all $n<m$ we have
\begin{equation*}
    \sum_{j=0}^{n-m} F_j^m F_{n-m-j}^{m+1} = 0 \quad\mbox{and}\quad F_n^{m+1} - F_n^m = 0.
\end{equation*}
\end{proof}

As a corollary, we obtain Theorem 2.1 from \cite{Frontczak2018} (identity \eqref{eq:T=F+FT} below) and many more.
\begin{corollary}\label{cor:convFT/TQ/QP}
Let $n\geq 0$ be an integer. Then
\begin{align}
	 \label{eq:T=F+FT} \sum\limits_{j=0}^{n} F_j T_{n-j} &= T_{n+2} - F_{n+2}, \\
   \label{eq:Q=T+TQ} \sum\limits_{j=0}^{n} T_j Q_{n-j} &= Q_{n+3} - T_{n+3},\\
   \label{eq:P=Q+QP} \sum\limits_{j=0}^{n} Q_j P_{n-j} &= P_{n+4} - Q_{n+4}.
\end{align}
\end{corollary}
\begin{proof}
   Set $m=2$, $m=3$ and $m=4$, respectively, in Theorem \ref{th:general_conv_m_step}.
\end{proof}

It is interesting that Theorem \ref{th:general_conv_m_step} can be generalized further. Let $p\geq 0$ be an integer. Then
we have the following result.

\begin{theorem}\label{thm:gen_m+p_conv}
Let $p\geq 0$ be an integer. For all $n\geq m$ we have the identity
\begin{equation}\label{eq:sumF^mF^m+p=F^m+p-F^m}
 \sum_{k=0}^{p-1} \sum_{j=0}^{n-m-k} F_j^{(m)} F_{n-m-k-j}^{(m+p)} = F_n^{(m+p)} - F_n^{(m)}.
\end{equation}
\end{theorem}
\begin{proof}
This follows from the relation
\begin{equation*}
    \frac{x}{F^{(m+p)}(x)} = \frac{x-\big (\sum_{k=1}^p x^{m+k}\big ) F^{(m)}(x)}{F^{(m)}(x)}
\end{equation*}
or equivalently
\begin{equation*}
    F^{(m+p)}(x) - F^{(m)}(x) = \Big (\sum_{k=1}^p x^{m+k-1}\Big ) F^{(m)}(x) F^{(m+p)}(x).
\end{equation*}
The remaining part of the proof is similar to the proof of Theorem \ref{th:general_conv_m_step}.
\end{proof}

When $p=1$ then we get Theorem \ref{th:general_conv_m_step}. When $p=2$ then we get the following.
\begin{corollary}\label{cor:Fm+2-Fm=mixed}
For all $n\geq m$ we have the following identity
\begin{equation*}
\sum_{j=0}^{n-m-1} F_j^{(m)} (F_{n-m-j}^{(m+2)} + F_{n-m-j-1}^{(m+2)}) = F_n^{(m+2)} - F_n^{(m)}.
\end{equation*}
\end{corollary}

Corollary \ref{cor:Fm+2-Fm=mixed} allows to obtain many different identities involving convolutions of two sequences.
We show the following two examples.
\begin{example}
Set $m=1$ in Corollary \ref{cor:Fm+2-Fm=mixed} to get
\begin{equation*}
T_n - 1 = \sum_{j=0}^{n-2} T_j + \sum_{j=0}^{n-3} T_j = 2 \sum_{j=0}^{n-3}T_j + T_{n-2}.
\end{equation*}
This gives (after replacing $n$ by $n+3$)
\begin{equation}\label{Tribonacci_sum}
\sum_{j=0}^{n} T_j = \frac{1}{2} (T_{n+3} - T_{n+1} - 1) = \frac{1}{2} (T_{n+2} + T_{n} - 1),
\end{equation}
which is a well-known partial sum formula (see for example \cite{BenBook, Frontczak2018b, Gryszka2022}).
\end{example}

\begin{example}
For any $n\geq 0$ we have
\begin{equation}\label{eq:FQ=Q-F}
  \sum_{j=0}^{n} F_{j+2} Q_{n-j} = Q_{n+3} - F_{n+3}.
\end{equation}
This follows from setting $m=2$ in Corollary \ref{cor:Fm+2-Fm=mixed} and calculating
\begin{align*}
Q_n - F_n &= \sum_{j=0}^{n-3} F_j (Q_{n-2-j}+Q_{n-3-j}) \\
&= \sum_{j=1}^{n-3} F_j Q_{n-2-j} + \sum_{j=0}^{n-3} F_j Q_{n-3-j} \\
&= \sum_{j=0}^{n-4} F_{j+1} Q_{n-3-j} + \sum_{j=0}^{n-3} F_j Q_{n-3-j} \\
&= \sum_{j=0}^{n-4} F_{j+2} Q_{n-3-j} \\
&= \sum_{j=0}^{n-3} F_{j+2} Q_{n-3-j}.
\end{align*}
\end{example}

Finally, we point out that when $m=1$ in Theorem \ref{thm:gen_m+p_conv} then we get the partial sum formula
for Fibonacci $m$-step numbers.
\begin{corollary}\label{cor:F_1+p-1=partialsum}
For any $n\geq p$,
    \begin{equation}\label{eq:Partial_sum_Fibonacci_m_step}
        F^{(p+1)}_n - 1 = \sum_{k=0}^{p-1}\sum\limits_{j=0}^{n-2-k} F_{j}^{(p+1)}.
    \end{equation}
\end{corollary}

In fact, rewriting \eqref{eq:Partial_sum_Fibonacci_m_step} to a slightly more convenient way we can get the following general identity.
\begin{theorem}[Partial sum formula for Fibonacci $m$-step numbers]\label{partial_sum}
   For any $m\geq 2$ we have
    \begin{equation}\label{eq:partial_sum_formuka}
        \sum\limits_{k=0}^n F_k^{(m)} = \frac{1}{m-1} \left(F_{n+m}^{(m)} - \sum_{k=1}^{m-2} k F_{n+k}^{(m)} - 1 \right).
    \end{equation}
\end{theorem}
\begin{proof}
    Substitute $p+1\to m$ in \eqref{eq:Partial_sum_Fibonacci_m_step} to get
    \begin{align*}
        F_n^{(m)} - 1 &= \sum\limits_{k=0}^{m-2} \sum\limits_{j=0}^{n-2-k} F_k^{(m)} \\
        &= (m-1)\sum\limits_{k=0}^{n-m} F_k^{(m)} + \sum\limits_{k=1}^{m-2} k F_{n-m+k}^{(m)}.
    \end{align*}
    Substituting $n\to n+m$ and rearranging the terms leads to the desired formula.
\end{proof}
We note that Theorem \ref{partial_sum} is not new. For instance, in \cite{Schumacher2019} the partial sum formula
is proved using induction but the proof is two pages long. A shorter proof of an equivalent version of the partial sum formula
is given by Dresden and Wang in \cite{Dresden1}. Second author has recently shown a proof without words of the identity in \cite{Gryszka2022}. Here, we obtained it as a corollary of a more general result.

We have seen from \eqref{eq:FQ=Q-F} that it is possible to derive a formula for the convolution of Fibonacci and Tetranacci numbers.
In the following, we gather all closed formulas for mixed convolutions of two Fibonacci $m$-step numbers with $2\leq m\leq 5$,
that are not present in Corollary \ref{cor:convFT/TQ/QP}.
\begin{corollary}\label{cor:conv_FQ_FP_TP}
Let $n\geq 0$ be an integer. Then
\begin{align}
 % \label{FQ=Q+Q-F} \sum\limits_{k=0}^{n} F_k Q_{n-k} &= Q_{n+1}+Q_{n-1}-F_{n+2}, \\
\label{FQ=Q+Q-F} \sum\limits_{j=0}^{n} F_j Q_{n-j} &= Q_{n+1} + Q_{n-1} - F_{n+1}, \\
 \label{FP=P+P=F} \sum\limits_{j=0}^{n} F_j P_{n-j} &= \frac{1}{2}\big( P_{n+2} + P_{n-1} - F_{n+2} \big), \\
 \label{PT=P+P-P-P-T-T+P+P}
						\sum\limits_{j=0}^{n} P_j T_{n-j} &= \frac{1}{2}\big( P_{n+3} + P_{n+1} + P_{n-1} - T_{n+3} - T_{n+1} \big).
\end{align}
\end{corollary}
\begin{proof}
    Identity \eqref{FQ=Q+Q-F} follows from \eqref{eq:FQ=Q-F} after fixing the summation range.
		To show \eqref{FP=P+P=F} we use Theorem \ref{thm:gen_m+p_conv} with $m=2$ and $p=3$. Then
    \begin{align*}
        P_n-F_n&=\sum\limits_{j=0}^{n-4}P_j F_{n-4-j} + \sum\limits_{j=0}^{n-3} P_j F_{n-3-j} + \sum\limits_{j=0}^{n-2} P_j F_{n-2-j} \\
        % &=\sum\limits_{k=0}^{n-4}P_kF_{n-4-k}+\sum\limits_{k=0}^{n-4}P_kF_{n-3-k}+\sum\limits_{k=0}^{n-4}P_kF_{n-2-k}\\
        % &\hspace{3cm} +P_{n-3}F_0+P_{n-2}F_0+P_{n-3}F_1\\
        &=2\sum\limits_{j=0}^{n-4} P_j F_{n-2-j} + P_{n-3}.
    \end{align*}
    It follows that
    \begin{equation*}
    \frac{1}{2}\big(P_{n+1}-P_{n-2}-F_{n+1}\big) = \sum\limits_{j=0}^{n-3} P_j F_{n-1-j} = \sum\limits_{j=0}^{n-1} P_j F_{n-1-j} - P_{n-2}
    \end{equation*}
    and this implies the identity.

    To show \eqref{PT=P+P-P-P-T-T+P+P} we use Theorem \ref{thm:gen_m+p_conv} with $p=2$ and $m=3$ and we get
    \begin{equation}\label{eq:P-T=PTT}
        P_n-T_n = \sum\limits_{j=0}^{n-4} P_j (T_{n-3-j} + T_{n-4-j})
    \end{equation}
    Then, using \eqref{eq:P-T=PTT} twice, second time with $n+2$ in place of $n$, and adding up we get
       $$P_{n+2}+P_n-T_{n+2}-T_n = 2\sum\limits_{j=0}^{n-4} P_j T_{n-1-j} + 2P_{n-3} + P_{n-2}.$$
    or, equivalently
    \begin{equation}\label{eq:P+P-P-2P-T-T=2PT}
        P_{n+2}+P_n-P_{n-2}-2P_{n-3}-T_{n+2}-T_n = 2\sum\limits_{j=0}^{n-4} P_j T_{n-1-j}.
    \end{equation}
    Notice that
    $$\sum\limits_{j=0}^{n-4} P_j T_{n-1-j} = \sum\limits_{j=0}^{n-1} P_j T_{n-1-j} - P_{n-2} - P_{n-3},$$
    thus substituting that to \eqref{eq:P+P-P-2P-T-T=2PT} and simplifying we get
    $$\sum\limits_{j=0}^{n-1} P_j T_{n-1-j} = \frac{1}{2}\big( P_{n+2} + P_{n} + P_{n-2} - T_{n+2} - T_{n} \big).$$
    This implies \eqref{PT=P+P-P-P-T-T+P+P}.
\end{proof}

The method used to obtain identities \eqref{FP=P+P=F} and \eqref{PT=P+P-P-P-T-T+P+P} will be explored in the general case in Section 4.

\subsection{Mixed convolutions with the "switch effect"}

In the following theorem we show a convolution-type formula where by switching a place of parentheses
and the minus sign with minor adjustment of indices, we obtain the equality. It is also important to mention
that the formula has an interesting connection to convolution sums of three sequences. The connection will be established
in the next two sections.

\begin{theorem}\label{main_thm3}
For all $n\geq 1$ and $m\geq 3$ we have
\begin{equation*}
\sum_{j=0}^{n}  F_j^{(m-2)} \big ( F_{n-j}^{(m)} - F_{n-j}^{(m-1)}\big )
= \sum_{j=0}^{n-1}  F_j^{(m)} \big ( F_{n-1-j}^{(m-1)} - F_{n-1-j}^{(m-2)}\big ).
\end{equation*}
\end{theorem}
\begin{proof}
Since
\begin{align*}
1-x-x^2-\cdots -x^m &= 1-x-x^2-\cdots -x^{m-1} - x (1-x-x^2-\cdots -x^{m-2}) \\
& \qquad + x (1-x-x^2-\cdots -x^{m-1})
\end{align*}
it follows that
\begin{equation*}
\frac{1}{F^{(m)}(x)} = \frac{1}{F^{(m-1)}(x)} - \frac{x}{F^{(m-2)}(x)} + \frac{x}{F^{(m-1)}(x)}
\end{equation*}
or equivalently
\begin{equation*}
F^{(m-2)}(x) \Big (F^{(m)}(x) - F^{(m-1)}(x)\Big ) = x F^{(m)}(x) \Big (F^{(m-1)}(x) - F^{(m-2)}(x)\Big ).
\end{equation*}
Passing to the power series and comparing coefficients of $x^n$ we obtain the identity.
\end{proof}

The next corollary is again a rediscovery of Theorem 2.1 in \cite{Frontczak2018}.

\begin{corollary}\label{cor1_thm3}
For all $n\geq 1$ we have
\begin{equation*}
T_{n+1} - F_{n+1} = \sum_{j=0}^{n-1} T_j F_{n-1-j}.
\end{equation*}
\end{corollary}
\begin{proof}
Set $m=3$ in Theorem \ref{main_thm3} and simplify using
\begin{equation*}
F_n^{(1)} = \begin{cases}
 0, &\text{\rm if $n=0$}; \\
 1, &\text{\rm if $n\geq 1$};
\end{cases}
\end{equation*}
equation \eqref{Tribonacci_sum} and
\begin{equation*}
\sum_{j=0}^n F_j = F_{n+2} - 1.
\end{equation*}
\end{proof}

\begin{corollary}\label{cor2_thm3}
For all $n\geq 1$ we have
\begin{equation*}
\sum_{j=0}^{n} F_{j}\big (T_{n-j} - Q_{n-j}\big ) = \sum_{j=0}^{n-1} \big ( F_{j} - T_{j}\big ) Q_{n-1-j}
\end{equation*}
and
\begin{equation}\label{eq:TQ-TP=TP-QP}
     \sum\limits_{j=0}^n T_j \big(Q_{n-j} - P_{n-j} \big) = \sum\limits_{j=0}^{n-1} \big(T_{j} - Q_{j}\big ) P_{n-1-j}.
  \end{equation}
\end{corollary}
\begin{proof}
Set $m=4$ and $m=5$, respectively, in Theorem \ref{main_thm3} and simplify.
\end{proof}
It is also possible to obtain the "switch" effect in the following sense.

\begin{theorem}\label{main_thm4}
For all $m,n\geq 2$ we have
\begin{equation*}
\sum_{j=0}^{n}  F_j^{(m-1)} F_{n-j}^{(m)} = \sum_{j=0}^{n-2} 2^j \big ( F_{n-1-j}^{(m-1)} - F_{n-2-j}^{(m)}\big ).
\end{equation*}
\end{theorem}
\begin{proof}
The relation
\begin{equation*}
1 - x - x^2 - \cdots - x^m = 1 - 2x + x (1-x-x^2-\cdots -x^{m-1})
\end{equation*}
translates to
\begin{equation*}
\frac{1}{F^{(m)}(x)} = \frac{1}{P_2(x)} + \frac{x}{F^{(m-1)}(x)}
\end{equation*}
with
\begin{equation*}
P_2(x) = \frac{x}{1-2x} = x \sum_{n=0}^\infty 2^n x^n.
\end{equation*}
This gives
\begin{equation*}
P_2(x) F^{(m-1)}(x) = F^{(m)}(x) F^{(m-1)}(x) + x P_2(x) F^{(m)}(x).
\end{equation*}
Passing to the power series and comparing coefficients of $x^n$ we obtain the identity.
\end{proof}

\begin{corollary}\label{cor1_thm4}
For all $n\geq 0$
\begin{equation}\label{eq:conv_2^j_F_n-j}
\sum_{j=0}^{n} 2^j F_{n-j} = 2^{n+1} - F_{n+3}.
\end{equation}
\end{corollary}
\begin{proof}
Set $m=2$ in Theorem \ref{main_thm4}, simplify, and replace $n$ by $n+2$.
\end{proof}

\begin{corollary}\label{cor2_thm4}
For all $n\geq 0$
\begin{equation*}
\sum_{j=0}^{n} 2^j \big ( F_{n+1-j} - T_{n-j} \big ) = T_{n+4} - F_{n+4}.
\end{equation*}
\end{corollary}
\begin{proof}
Set $m=3$ in Theorem \ref{main_thm4}, use Corollary \ref{cor1_thm3}, and replace $n$ by $n+2$.
\end{proof}

\begin{corollary}\label{cor3_thm4}
For all $n\geq 0$
\begin{equation}\label{eq:conv_2^j_T_n-j}
\sum_{j=0}^{n} 2^j T_{n-j} = 2^{n+2} - T_{n+4}.
\end{equation}
\end{corollary}
\begin{proof}
Combine Corollary \ref{cor1_thm4} with Corollary \ref{cor2_thm4}.
\end{proof}

\begin{corollary}\label{cor4_thm4}
For all $n\geq 0$
\begin{equation*}
\sum_{j=0}^{n+2} T_j Q_{n-j} = \sum_{j=0}^{n} 2^j \big ( T_{n+1-j} - Q_{n-j} \big ).
\end{equation*}
\end{corollary}
\begin{proof}
Set $m=4$ in Theorem \ref{main_thm4}, and replace $n$ by $n+2$.
\end{proof}

% \begin{corollary}\label{cor4_thm4}
% For all $n\geq 0$
% \begin{equation}
% \sum_{j=0}^{n} 2^j Q_{n-j} = 2^{n+3} - Q_{n+5}.
% \end{equation}
% \end{corollary}
% \begin{proof}
% Combine Corollary \ref{cor3_thm4} with the first part of Corollary \ref{cor2_thm_gen_conv}.
% \end{proof}

This is a general formula of \eqref{eq:conv_2^j_F_n-j} and \eqref{eq:conv_2^j_T_n-j} for Fibonacci $m$-step numbers.
\begin{theorem}\label{power2_conv}
For all $n\geq 0$ and $m\geq 1$
\begin{equation*}
\sum_{j=0}^{n} 2^j F_{n-j}^{(m)} = 2^{n-1+m} - F_{n+1+m}^{(m)}.
\end{equation*}
\end{theorem}
\begin{proof}
We use induction on $m$. The statement is true for $m=1$ and $m=2$. Assume the identity is true for a fixed $m-1>2$ (and all $n$).
Replacing $n$ by $n+2$ in Theorem \ref{main_thm4} yields
\begin{equation*}
\sum_{j=0}^{n} 2^j F_{n-j}^{(m)} = \sum_{j=0}^{n} 2^j F_{n+1-j}^{(m-1)} - \sum_{j=0}^{n+2}  F_j^{(m-1)} F_{n+2-j}^{(m)}.
\end{equation*}
But from Theorem \ref{th:general_conv_m_step} upon making the replacements $n\mapsto n+2+m$ and $m\mapsto m-1$ we get
\begin{equation*}
\sum_{j=0}^{n+2}  F_j^{(m-1)} F_{n+2-j}^{(m)} = F_{n+1+m}^{(m)} - F_{n+1+m}^{(m-1)}.
\end{equation*}
From here using the inductive hypothesis we can calculate
\begin{align*}
\sum_{j=0}^{n} 2^j F_{n-j}^{(m)} &= \sum_{j=0}^{n} 2^j F_{n+1-j}^{(m-1)} - F_{n+1+m}^{(m)} + F_{n+1+m}^{(m-1)} \\
&= \sum_{j=0}^{n+1} 2^j F_{n+1-j}^{(m-1)} - F_{n+1+m}^{(m)} + F_{n+1+m}^{(m-1)} \\
&= 2^{n+1-1+(m-1)} - F_{n+1+1+(m-1)}^{(m-1)} - F_{n+1+m}^{(m)} + F_{n+1+m}^{(m-1)} \\
&= 2^{n-1+m} - F_{n+1+m}^{(m)}.
\end{align*}
\end{proof}

\subsection{Three other general identities}

The next convolution identities involve an alternating sum. They all share the same structure,
we convolve two sequences with indices $m$ that differ by $2$.

\begin{theorem}\label{alternating_thm}
For any $n\geq 2m-1$ we have
  \begin{equation}\label{alternating_conv}
	\sum\limits_{j=0}^{n-1} (-1)^j \big ( F_j^{(2m)} - F_{j}^{(2m-2)}\big )
	= (-1)^{n+1} \sum\limits_{j=0}^{n-2m+1} F_{j}^{(2m)}F_{n-2m+1-j}^{(2m-2)}.
   \end{equation}
\end{theorem}
\begin{proof}
As
\begin{equation*}
1 + x - x^2 + x^3 \mp \cdots - x^{2(m-1)} + x^{2m-1} - x^{2m} = 1 + x - x^2 + x^3 \mp \cdots - x^{2(m-1)} + x^{2m-1}(1-x),
\end{equation*}
we obtain the functional equation
\begin{equation*}
\frac{1}{F^{(2m)}(-x)} = \frac{1}{F^{(2m-2)}(-x)} - x^{2m-1}\frac{1}{F^{(1)}(x)}
\end{equation*}
or equivalently
\begin{equation*}
\Big ( F^{(2m)}(-x) - F^{(2m-2)}(-x) \Big ) F^{(1)}(x) = x^{2m-1} F^{(2m)}(-x) F^{(2m-2)}(-x).
\end{equation*}
When simplifying we again use
\begin{equation*}
F_n^{(1)} = \begin{cases}
 0, &\text{\rm if $n=0$}; \\
 1, &\text{\rm if $n\geq 1$}.
\end{cases}
\end{equation*}
\end{proof}
As a corollary, we can show a different proof of identity \eqref{FQ=Q+Q-F}.
\begin{corollary}\label{cor1_alt_thm}
For any $n\geq 0$ we have
\begin{equation}
\sum_{j=0}^n F_{j} Q_{n-j} = Q_{n+1} + Q_{n-1} - F_{n+1}.
\end{equation}
\end{corollary}
\begin{proof}
Set $m=2$ in Theorem \ref{alternating_thm} and use
\begin{equation*}
\sum_{j=0}^n (-1)^j F_j = (-1)^n F_{n-1} - 1
\end{equation*}
as well as \cite{Soykan}
\begin{equation*}
\sum_{j=0}^n (-1)^j Q_j = (-1)^n (Q_{n+3} - 2Q_{n+2} + Q_{n+1} - Q_n) - 1.
\end{equation*}
These results produce
\begin{equation*}
\sum_{j=0}^{n-3} F_{j} Q_{n-3-j} = Q_{n+2} - 2Q_{n+1} + Q_n - Q_{n-1} - F_{n-2},
\end{equation*}
and the statement follows upon replacing $n$ by $n+3$ and simplifying.
\end{proof}

The companion result for Theorem \ref{alternating_thm} for odd $m$ is stated next.

\begin{theorem}\label{alternating_thm2}
For any $n\geq 2m$ we have
  \begin{equation}\label{alternating_conv2}
	\sum\limits_{j=0}^{n-1} (-1)^{j+1} \big ( F_j^{(2m+1)} - F_{j}^{(2m-1)}\big )
	= (-1)^{n} \sum\limits_{j=0}^{n-2m} F_{j}^{(2m+1)}F_{n-2m-j}^{(2m-1)}.
   \end{equation}
\end{theorem}
\begin{proof}
This result follows from the functional equation
\begin{equation*}
\frac{1}{F^{(2m+1)}(-x)} = \frac{1}{F^{(2m-1)}(-x)} + x^{2m}\frac{1}{F^{(1)}(x)}
\end{equation*}
or equivalently
\begin{equation*}
\Big ( F^{(2m-1)}(-x) - F^{(2m+1)}(-x) \Big ) F^{(1)}(x) = x^{2m} F^{(2m-1)}(-x) F^{(2m+1)}(-x).
\end{equation*}
\end{proof}
We proceed with two corollaries. The first identity is known (see for instance Equation (18) in \cite{Frontczak2018b}),
the second identity is a rediscovery of \eqref{PT=P+P-P-P-T-T+P+P} with a slightly different (but equivalent) right hand side.

\begin{corollary}\label{cor1_alt_thm2}
For any $n\geq 0$ we have
\begin{equation*}
\sum_{j=0}^n (-1)^j T_j = \frac{1}{2} ((-1)^n (T_{n+1} - T_{n-1}) - 1).
\end{equation*}
\end{corollary}
\begin{proof}
Set $m=1$ in Theorem \ref{alternating_thm2} and simplify.
\end{proof}

\begin{corollary}\label{cor2_alt_thm2}
For any $n\geq 0$ we have
\begin{equation*}
\sum_{j=0}^{n} P_j T_{n-j} = \frac{1}{2} (-P_{n+7}+2P_{n+6}-P_{n+5}+2P_{n+4}+P_{n+3}-T_{n+4}+T_{n+2}).
\end{equation*}
\end{corollary}
\begin{proof}
Set $m=2$ in Theorem \ref{alternating_thm2} and simplify while making use of (see \cite{Soykan2})
\begin{equation*}
\sum_{j=0}^n (-1)^j P_j = \frac{1}{2} ((-1)^n (-P_{n+4} + 2P_{n+3} - P_{n+2} + 2P_{n+1} + P_n) - 1).
\end{equation*}
\end{proof}

We conclude this section with the remark that identities involving Fibonacci $m$-step numbers and
other important number sequences can be obtained fairly easily using the generating function approach.
We give an example involving Jacobsthal numbers, which is related to identity \eqref{eq:conv_sum_JF}.

\begin{theorem}\label{Jacobsthal_thm}
Let $J_{n}$ be the Jacobsthal numbers, i.e., $J_0=0, J_1=1$, and $J_{n+2}=J_{n+1}+2J_{n}$.
Then, for $m\geq 3$ and any $n\geq 2$ we have
  \begin{equation}\label{Jacob1}
     \sum\limits_{j=0}^n F_j^{(m)} F_{n-j}^{(m-2)} = J_{n-1} + \sum\limits_{j=0}^{n-2} J_j \big (F_{n-j}^{(m-2)} - F_{n-2-j}^{(m)} \big ).
   \end{equation}
\end{theorem}
\begin{proof}
We have
\begin{equation*}
 1 - x - x^2 - x^3 - \cdots - x^m = 1 - x - 2x^2 + x^2(1 - x - x^2 - \cdots - x^{m-2}),
\end{equation*}
which implies
\begin{equation*}
\frac{x}{F^{(m)}(x)} = \frac{x}{J(x)} + \frac{x^3}{F^{(m-2)}(x)},
\end{equation*}
where $J(x)=\frac{x}{1-x-2x^2}$ is the ordinary generating function for Jacobsthal numbers.
\end{proof}

\begin{corollary}\label{J+T-Conv}
For any $n\geq 0$ we have
\begin{equation*}
\sum_{j=0}^n J_j T_{n-j} = J_{n+1} + \frac{1}{2} (J_{n+2} - T_{n+3} - T_{n+1}).
\end{equation*}
\end{corollary}
\begin{proof}
Use the previous theorem with $m=3$ in conjunction with \eqref{Tribonacci_sum} and
\begin{equation*}
\sum_{j=0}^n J_j = \frac{1}{2} (J_{n+2} - 1).
\end{equation*}
\end{proof}

\section{Convolutions of multiple sequences}

The main goal of this section is to derive convolution identities of three and more Fibonacci $m$-step sequences.
For the convenience, we introduce the following notation:
$$K(\ell,b)=\{(k_1,\ldots,k_\ell)\in\mathbb{Z}^\ell_{\geq 0}\colon k_1+\cdots+k_\ell=b\}.$$

Before providing an analysis of the $m$-step sequences, we present a Pell sequence result. Then we consider a general convolution of three Fibonacci $m$-step sequences and derive all mixed convolutions with $2\leq m\leq 5$. Finally, we delve with the convolution of four sequences.

\subsection{A Pell-Fibonacci relation}

\begin{theorem}\label{P-Fib-conv}
Let $\mathcal{P}_{n}$ be the Pell numbers, i.e., $\mathcal{P}_0=0, \mathcal{P}_1=1$,
and $\mathcal{P}_{n+2}=2\mathcal{P}_{n+1}+\mathcal{P}_{n}$. Then, for all $n\geq 1$ and $m\geq 2$ we have the following
Pell-Fibonacci-$m$-step-relation:
\begin{equation*}
\sum_{K(3,n-1)} \mathcal{P}_{k_1} F_{k_2}^{(m-1)} F_{k_3}^{(m)}
= \sum_{j=0}^n F_j^{(m-1)} \left ( \mathcal{P}_{n-j} - F_{n-j}^{(m)}\right ) - \sum_{j=0}^{n-1} \mathcal{P}_j F_{n-1-j}^{(m)}.
\end{equation*}
\end{theorem}
\begin{proof}
Let $\mathcal{P}(x)=\frac{x}{1-2x-x^2}$ denote the generating function for the Pell numbers. From
\begin{equation*}
 1 - x - x^2 - x^3 - \cdots - x^m = 1 - 2x - x^2 + x(1 - x - x^2 - \cdots - x^{m-1}) + x^2,
\end{equation*}
we get
\begin{equation*}
\frac{1}{F^{(m)}(x)} = \frac{1}{\mathcal{P}(x)} + x \frac{F^{(m-1)}(x)+1}{F^{(m-1)}(x)},
\end{equation*}
or
\begin{equation*}
x \mathcal{P}(x) F^{(m-1)}(x) F^{(m)}(x) = F^{(m-1)}(x)\left ( \mathcal{P}(x) - F^{(m)}(x)\right ) - x \mathcal{P}(x) F^{(m)}(x).
\end{equation*}
The result follows upon passing to power series and comparing the coefficients of $x^n$.
\end{proof}

As a first corollary, we rediscover identity \eqref{P-F-conv}.
\begin{corollary}
For any $n\geq 0$ we have
\begin{equation*}
\sum_{j=0}^n \mathcal{P}_j F_{n-j} = \mathcal{P}_{n} - F_{n}.
\end{equation*}
\end{corollary}
\begin{proof}
Set $m=2$ in Theorem \ref{P-Fib-conv} and simplify.
\end{proof}

\begin{corollary}
For any $n\geq 0$ we have
\begin{equation*}
\sum_{K(3,n)} \mathcal{P}_{k_1} F_{k_2} T_{k_3} = \mathcal{P}_{n+1} + F_{n+2} - T_{n+3} - \sum_{j=0}^{n} \mathcal{P}_j T_{n-j}.
\end{equation*}
or
\begin{equation}\label{Pell-F-T}
\sum_{K(3,n)} \mathcal{P}_{k_1} F_{k_2} T_{k_3} = \frac{1}{2}\left ( \mathcal{P}_{n+1} - T_{n+3} - T_{n+2}\right ) + F_{n+2}.
\end{equation}
\end{corollary}
\begin{proof}
Set $m=3$ in Theorem \ref{P-Fib-conv}, use \eqref{F-T-conv} and \eqref{P-F-conv}, and simplify.
To get \eqref{Pell-F-T}, let $R(x)=\frac{x}{1-x^2}$ and notice that
$$\frac{1}{T(x)}=\frac{1}{\mathcal{P}(x)}+x\cdot\frac{1}{R(x)},$$
which is equivalent to
$$x\mathcal{P}(x)T(x)=R(x)(\mathcal{P}(x)-T(x)).$$
Using
$$R(x) = \frac{x}{2}\left(\frac{1}{1-x}+\frac{1}{1+x}\right)$$
and passing to power series, we get
\begin{equation*}
\sum_{j=0}^{n} \mathcal{P}_j T_{n-j} = \frac{1}{2}\sum_{j=0}^n \mathcal{P}_j + \frac{1}{2}\sum_{j=0}^n (-1)^{n-j} \mathcal{P}_j
- \frac{1}{2}\sum_{j=0}^n T_j - \frac{1}{2}\sum_{j=0}^n (-1)^{n-j} T_j.
\end{equation*}
Recall the well-known identities for Pell numbers (see for example \cite{Bradie2010}):
\begin{equation*}
\sum_{j=0}^n \mathcal{P}_j = \frac{1}{2} \left ( \mathcal{P}_{n+1} + \mathcal{P}_n - 1\right ) \quad\mbox{and}\quad
\sum_{j=0}^n (-1)^j \mathcal{P}_j = \frac{1}{2} \left ((-1)^n ( \mathcal{P}_{n+1} - \mathcal{P}_n) - 1\right ).
\end{equation*}

These relations show that
\begin{equation*}
\sum_{j=0}^{n} \mathcal{P}_j T_{n-j} = \frac{1}{2}\left ( \mathcal{P}_{n+1} - T_{n+1} - T_{n}\right )
\end{equation*}
and \eqref{Pell-F-T} follows.
\end{proof}

From
\begin{equation*}
\mathcal{P}(x) F(x) = \mathcal{P}(x) - F(x),
\end{equation*}
we get (by induction) for all $r\geq 1$

\begin{equation*}
\mathcal{P}(x) F^r(x) = \mathcal{P}(x) - \sum_{s=1}^r F^s(x),
\end{equation*}
or equivalently
\begin{equation*}
\sum_{K(r+1,n)} \mathcal{P}_{k_1} F_{k_2} F_{k_3} \cdots F_{k_{r+1}} = \mathcal{P}_n - \sum_{s=1}^r \sum_{K(s,n)} {F}_{k_1} \cdots F_{k_{s}}.
\end{equation*}

Special cases of this convolution include (see Zhang's paper \cite{Zhang} for the Fibonacci convolutions)
\begin{equation*}
\sum_{K(3,n)} \mathcal{P}_{k_1} F_{k_2} F_{k_3} = \mathcal{P}_n - F_n - \frac{1}{5}\left ( (n-1)F_n + 2n F_{n-1}\right) \quad (n\geq 1)
\end{equation*}
as well as
\begin{align}
\sum_{K(4,n)} \mathcal{P}_{k_1} F_{k_2} F_{k_3} F_{k_4} &= \mathcal{P}_n - F_n - \frac{1}{5}\left ( (n-1)F_n + 2n F_{n-1}\right) \nonumber \\
& \quad - \frac{1}{50}\left ((5n^2-9n-2) F_{n-1} + (5n^2-3n-2) F_{n-2}\right ) \quad (n\geq 2).
\end{align}

\subsection{Convolution of three Fibonacci $m$-step sequences}

The next theorem is the main observation and will be used as a basis for some observations what will follow.

\begin{theorem}\label{main_thm_mult}
Let $m,p,q\geq 1$ be integers. Then the following functional identity holds true:
\begin{align}
& x^{2m+p} (1-x^p)(1-x^q)F^{(m)}(x)F^{(m+p)}(x)F^{(m+p+q)}(x) \nonumber \\
\qquad &= x^m (1-x)(1-x^p) F^{(m)}(x)F^{(m+p+q)}(x) - (1-x)^2 (F^{(m+p)}(x) - F^{(m)}(x)).
\end{align}
In particular,
\begin{equation}\label{eq:FmFm+1Fm+2=FmFM+2-Fm+1+Fm}
x^{2m+1} F^{(m)}(x)F^{(m+1)}(x)F^{(m+2)}(x) = x^m F^{(m)}(x)F^{(m+2)}(x) - F^{(m+1)}(x) + F^{(m)}(x).
\end{equation}
\end{theorem}
\begin{proof}
We utilize the proof of Theorem \ref{thm:gen_m+p_conv} and work with the identity:
  \begin{equation}\label{Eq:Fab-Fa=xFaFab}
      F^{(a+b)}(x)-F^{(a)}(x)=\left(\sum\limits_{k=0}^{b-1}x^{a+k}\right)F^{(a)}(x)F^{(a+b)}(x).
  \end{equation}
Applying it twice we get:
   \begin{align}
    F^{(m+p)}(x) &= \frac{F^{(m)}(x)}{1-\left(\sum\limits_{k=0}^{p-1}x^{m+k}\right)F^{(m)}(x)}, \nonumber\\
    F^{(m+p+q)}(x)-F^{(m+p)}(x) &= \left(\sum\limits_{k=0}^{q-1}x^{m+p+k}\right)F^{(m+p)}(x)F^{(m+p+q)}(x).\label{eq:sec3_eq1}
    \end{align}
Hence, we obtain
   \begin{equation*}
    F^{(m+p+q)}(x)-\frac{F^{(m)}(x)}{1-\left(\sum\limits_{k=0}^{p-1}x^{m+k}\right)F^{(m)}(x)}
		=\left(\sum\limits_{k=0}^{q-1}x^{m+p+k}\right)F^{(m+p)}(x)F^{(m+p+q)}(x),
    \end{equation*}
or, after rearranging,
  \begin{align*}
  & \left(\sum\limits_{k=0}^{p-1} x^{m+k}\right) \left(\sum\limits_{k=0}^{q-1} x^{m+p+k}\right) F^{(m)}(x)F^{(m+p)}(x)F^{(m+p+q)}(x) \\
	& \qquad = \left(\sum\limits_{k=0}^{q-1} x^{m+p+k}\right) F^{(m+p)}(x)F^{(m+p+q)}(x) \\
  & \qquad\qquad + \left(\sum\limits_{k=0}^{p-1} x^{m+k}\right) F^{(m)}(x)F^{(m+p+q)}(x) - F^{(m+p+q)}(x) + F^{(m)}(x).
  \end{align*}
Simplifying using \eqref{eq:sec3_eq1} yields
  \begin{align*}
  & \left(\sum\limits_{k=0}^{p-1} x^{m+k}\right) \left(\sum\limits_{k=0}^{q-1} x^{m+p+k}\right) F^{(m)}(x)F^{(m+p)}(x)F^{(m+p+q)}(x) \\
	& \qquad = \left(\sum\limits_{k=0}^{p-1} x^{m+k}\right) F^{(m)}(x)F^{(m+p+q)}(x) - F^{(m+p)}(x) + F^{(m)}(x).
  \end{align*}
Finally, from the geometric series
\begin{equation*}
\left(\sum_{k=0}^{p-1} x^{m+k}\right) \left(\sum_{k=0}^{q-1} x^{m+p+k}\right) = x^{2m+p} \frac{(1-x^p)(1-x^q)}{(1-x)^2}, \quad
\sum_{k=0}^{p-1} x^{m+k}= x^m \frac{1-x^p}{1-x}
\end{equation*}
and the functional equation follows. The particular case belongs to $p=q=1$.
\end{proof}

The special case of Theorem \ref{main_thm_mult} leads to yet another identity involving triple convolutions.
Namely, we have the following.
\begin{theorem}\label{th:x^mF^mF^m+1F^p}
    For any $p\geq 1$ and any $m\geq 1$ we have
		\begin{equation*}
    x^mF^{(m)}(x)F^{(m+1)}(x)F^{(p)}(x)=F^{(p)}(x)F^{(m+1)}(x)-F^{(p)}(x)F^{(m)}(x).
		\end{equation*}
    In particular,
    \begin{equation*}
		\sum\limits_{K(3,n-m)}F^{(m)}_{k_1}F^{(m+1)}_{k_2}F^{(p)}_{k_3}
		= \sum\limits_{j=0}^{n}F^{(p)}_j F^{(m+1)}_{n-j} - \sum\limits_{j=0}^n F^{(p)}_j F^{(m)}_{n-j}.
		\end{equation*}
\end{theorem}
\begin{proof}
   Use identity \eqref{Eq:Fab-Fa=xFaFab}.
\end{proof}

We apply Theorem \ref{main_thm_mult} with $m=2$ and $p=q=1$. This choice results in several different functional equations,
the penultimate one coming directly from the theorem.
\begin{align*}
x^5 F(x) T(x) Q(x) &= x^2 F(x) x^3 T(x) Q(x) \\
&= x^2 F(x) (Q(x) - T(x)) \\
&= x^2 F(x) Q(x) - T(x) + F(x) \\
&= x^2 F(x) Q(x) + x^3 T(x) Q(x) - Q(x) + F(x).
\end{align*}

\begin{theorem}\label{thrm:FTQ=Q+Q-T+F}
  We have for each $n\geq 0$,
  \begin{equation}\label{eq:FTQ=Q+Q-T+F}
    \sum\limits_{K(3,n)} F_{k_1} T_{k_2} Q_{k_3} = Q_{n+4} + Q_{n+2} - T_{n+5} + F_{n+3}.
  \end{equation}
\end{theorem}
\begin{proof}
Work with
$$x^5 F(x) T(x) Q(x) = x^2 F(x) Q(x) -T(x)+F(x).$$
When passing to the power series and comparing the coefficients of $x^n$ use the identity \eqref{FQ=Q+Q-F} and simplify.
\end{proof}

Working with $m=3$ and $p=q=1$ in Theorem \ref{main_thm_mult} results in the functional equations, where again,
the penultimate one coming directly from the theorem.
\begin{align*}
x^7 T(x) Q(x) P(x) &= x^3 T(x) x^4 Q(x) P(x) \\
&= x^3 T(x) (P(x) - Q(x)) \\
&= x^3 T(x) P(x) - Q(x) + T(x) \\
&= x^3 T(x) P(x) + x^4 P(x) Q(x) - P(x) + T(x),
\end{align*}
and this gives the next convolution.

\begin{theorem}
    We have
		\begin{equation*}
    \sum_{K(3,n-7)} T_{k_1} Q_{k_2} P_{k_3} = \frac{1}{2}(P_n + P_{n-2} + P_{n-4} + T_n - T_{n-2}) - Q_n
		\end{equation*}
    valid for each $n\geq 7$.
\end{theorem}
\begin{proof}
    Using the identity
    $$x^7 T(x) Q(x) P(x)=x^3 T(x) P(x) - Q(x) + T(x)$$
    in conjunction with \eqref{PT=P+P-P-P-T-T+P+P} we obtain, after minor simplification, the desired result.
\end{proof}

In the following two results we find convolution sums of the remaining triples with $2\leq m\leq 5$ using Theorem \ref{th:x^mF^mF^m+1F^p}.
\begin{theorem}\label{th:FQP_conv}
    We have the following identity:
    \begin{equation}\label{eq:conv_FQP}
		\sum\limits_{K(3,n-5)} F_{k_1} Q_{k_2} P_{k_3} = - Q_n - Q_{n-2} + \frac{1}{2}\big(P_{n+1} + P_{n-2} - F_{n+1} \big) + F_n.
		\end{equation}
\end{theorem}
\begin{proof}
Use Theorem \ref{th:x^mF^mF^m+1F^p} with $m=4$ and $p=2$, pass to power series, apply \eqref{FQ=Q+Q-F} and \eqref{FP=P+P=F}, and simplify.
\end{proof}

\begin{theorem}\label{th:FTP_conv}
    We have the following identity:
    $$\sum\limits_{K(3,n-5)}F_{k_1}T_{k_2}P_{k_3}=\frac{1}{2}\big(P_n-P_{n-1}+P_{n-2}-T_n-T_{n-2}+F_{n-1}\big).$$
\end{theorem}
\begin{proof}
Use Theorem \ref{th:x^mF^mF^m+1F^p} with $m=2$ and $p=5$, pass to power series, apply \eqref{FP=P+P=F} and \eqref{PT=P+P-P-P-T-T+P+P},
and simplify.
\end{proof}

% In the proofs of Theorem \ref{th:FQP_conv} and Theorem \ref{th:FTP_conv} we reduce the polynomial appearing in the product
% of three generating functions to a monomial by utilizing the same trick twice. This is necessary here, since, otherwise,
% the obtained formula would produce an identity with multiple triple convolution terms instead of one. However, we can avoid
% the presence of polynomial with at least two terms next to the product of three generating functions if we utilize other relations
% between generating functions.
In the proofs of Theorem \ref{th:FQP_conv} and Theorem \ref{th:FTP_conv} we used different approach.
Instead of utilizing Theorem \ref{main_thm_mult}, we applied Theorem \ref{th:x^mF^mF^m+1F^p}. This is necessary here,
since, otherwise, the obtained formula would produce an identity with multiple triple convolution terms instead of one.

We finish this part of the article by showing that one can use a completely different approach to prove the results in this section.
Namely, we can deduce the fundamental functional equations from other relations of the respective generating functions.

\begin{remark}
	Notice that one can prove Theorem \ref{th:FQP_conv} and Theorem \ref{th:FTP_conv} using other transformations of generating functions.
	For the first theorem, we can use the following identities:
    $$P(x)-Q(x)=x^4P(x)Q(x)\qquad \text{and}\qquad Q(x)=\frac{F(x)}{1+x^2-xF(x)},$$
    where the formula for $Q$ is substituted to the right-hand-side presence of $Q$ in the first identity.
  For the second theorem, we start from the identity
    $$\frac{x}{P(x)}=\frac{x}{F(x)}+\frac{x^4}{F(x)}-2x^3$$
    to obtain
    \begin{equation}\label{eq:P-F=2x2FP-x3P}
        P(x)-F(x)=2x^2F(x)P(x)-x^3P(x).
    \end{equation}
    Then we substitute $F(x)=\frac{T(x)}{1+x^2T(x)}$ and obtain a functional equation.
\end{remark}

\subsection{Convolutions of four Fibonacci $m$-step sequences}

The functional equations from Theorem \ref{main_thm_mult} (as well as any other functional equations derived in this paper)
can be multiplied by arbitrary $F^{(r)}(x)$ or even by any product of the form
\begin{equation*}
\prod_{u=1}^N \left (F^{(r_u)}(x)\right )^{s_u}, \qquad s_1,\ldots,s_N, r_1,\ldots,r_N \,\,\mbox{positive integers}.
\end{equation*}
In particular cases these functional equations can be used for a different proof of Theorems \ref{thrm:FTQ=Q+Q-T+F}--\ref{th:FTP_conv}.
For example, using \eqref{eq:FmFm+1Fm+2=FmFM+2-Fm+1+Fm} we immediately get the following.
\begin{theorem}
For any $m\geq 1$ we have
\begin{align}
& x^{2m+1} F^{(m)}(x)F^{(m+1)}(x)F^{(m+2)}(x) F^{(m+3)}(x)  \nonumber \\
& \qquad \qquad = x^m F^{(m)}(x)F^{(m+2)}(x) F^{(m+3)}(x) - F^{(m+1)}(x)F^{(m+3)}(x) + F^{(m)}(x)F^{(m+3)}(x).\label{eq:quad_conv_consecutive}
\end{align}
\end{theorem}

This theorem implies a particular closed formula for the convolution of Fibonacci, Tribonacci, Tetranacci and Pentanacci numbers.
\begin{theorem}
  We have
	\begin{equation*}
  \sum\limits_{K(4,n-5)}F_{k_1}T_{k_2}Q_{k_3}P_{k_4}=\frac{1}{2}\big(P_{n+4}-P_{n+3}+P_{n+2}+T_{n+3}+T_{n+1}-F_n\big)-Q_{n+3}-Q_{n+1}.
	\end{equation*}
\end{theorem}
\begin{proof}
    Set $m=2$ in Equation \eqref{eq:quad_conv_consecutive} to get
    $$x^5F(x)T(x)Q(x)P(x)=x^2F(x)Q(x)P(x)-T(x)P(x)+F(x)P(x).$$
    Pass to power series, use \eqref{eq:conv_FQP}, \eqref{PT=P+P-P-P-T-T+P+P} and \eqref{FP=P+P=F} and simplify.
\end{proof}

\begin{remark}
We note that identity \eqref{eq:quad_conv_consecutive} can be (after multiplication by $x^2$) further simplified
using Theorem \ref{th:x^mF^mF^m+1F^p} so that the right-hand-side contains convolution sums of two sequences.
\end{remark}

Other interesting identities involving Fibonacci, Tribonacci, Tetranacci and Pentanacci numbers can be derived.
To be particular, the following hold true.
\begin{align*}
x^4 F(x) T(x) Q(x) P(x) &= F(x)T(x) x^4Q(x)P(x) \\
&=F(x) T(x) (P(x) - Q(x)) \\
&=F(x) T(x) P(x) - x F(x) T(x) Q(x),
\end{align*}
or
\begin{align}
x^6 F(x) T(x) Q(x) P(x) = x^2 F(x)T(x) x^4 Q(x)P(x)= (T(x) - F(x)) (P(x) - Q(x)).\label{eq:FTQP=(T-F)(P-Q)}
\end{align}
The last identity leads to the following, surprising identity.

\begin{theorem}\label{thrm:conv_quad_()()}
    For any $m\geq 1$ we have
   \begin{equation*}
	\sum\limits_{K(4,n-2m-2)}F^{(m)}_{k_1}F^{(m+1)}_{k_2}F^{(m+2)}_{k_3}F^{(m+3)}_{k_3}
		=\sum\limits_{j=0}^n\big(F^{(m+3)}_j-F^{(m+2)}_j\big)\big(F^{(m+1)}_{n-j}-F^{(m)}_{n-j}\big).
	\end{equation*}
\end{theorem}
\begin{proof}
Following identity \eqref{eq:FTQP=(T-F)(P-Q)} and using \eqref{eq:F^m+1-F^m-x^mF^m+1F^m} we have
\begin{align*}
    &x^{2m+2}F^{(m)}(x)F^{(m+1)}(x)F^{(m+2)}(x)F^{(m+3)}(x) \\
    &\qquad\qquad=\big(x^mF^{(m)}(x)F^{(m+1)}(x)\big)\big(x^{m+2}F^{(m+2)}(x)F^{(m+3)}(x)\big)\\
    &\qquad\qquad=\big(F^{(m+1)}(x)-F^{(m)}(x)\big)\big(F^{(m+1)}(x)-F^{(m)}(x)\big)
\end{align*}
and the identity appears after passing to power series.
\end{proof}

\section{Concluding results and comments}

\subsection{Higher order convolutions}

Identities for even higher order convolutions can also be derived by mimicking the argument concluding previous section.
For instance, Theorem~\ref{thrm:conv_quad_()()} has the following generalization.
\begin{theorem}\label{th:general_conv_lots_of_seq}
  Fix $\ell\geq 1$ and an integer $m\geq 1$. Then
   \begin{equation}\label{eq:general_conv_v2_even}
   \sum\limits_{K(2\ell,n-\ell(m+\ell-1))}\prod\limits_{j=0}^{2\ell-1} F_{k_{j+1}}^{(m+j)}
	 = \sum\limits_{K(\ell,n)}\prod\limits_{j=0}^{\ell-1} \big(F^{(m+2j+1)}_{k_{j}}-F^{(m+2j)}_{k_{j}}\big)
   \end{equation}
   and
    \begin{equation}\label{eq:general_conv_v2_odd}
    \sum\limits_{K(2\ell+1,n-\ell(m+\ell-1))}\prod\limits_{j=0}^{2\ell} F_{k_j}^{(m+j)}=
    \sum\limits_{K(\ell+1,n)}F_{k_{\ell+1}}^{(m+2\ell)}\left[\prod\limits_{j=0}^{\ell-1} \big(F^{(m+2j+1)}_{k_{j}}-F^{(m+2j)}_{k_{j}}\big)\right].
    \end{equation}
\end{theorem}
\begin{proof}
    We sketch the proof of \eqref{eq:general_conv_v2_even}. Using \eqref{Eq:Fab-Fa=xFaFab}, we can write the following functional equation.
    \begin{align*}
        &x^{\ell(m+\ell-1)}F^{(m)}(x)\cdots F^{(m+2\ell-1)}(x)\\
        &\quad =\big(x^m F^{(m)}(x)F^{(m+1)}(x)\big)\cdots x^{m+2\ell-2}\big(x^m F^{(m+2\ell-2)}(x)F^{(m+2\ell-1)}(x)\big)\\
        &\quad=\big(F^{(m+1)}(x)-F^{(m)}(x)\big)\cdots \big(F^{(m+2\ell-1)}(x)-F^{(m+2\ell-2)}(x)\big)
    \end{align*}
    Passing to power series we obtain the desired formula.
\end{proof}
Theorem \ref{th:general_conv_lots_of_seq}, in  particular, allows us to effectively cut off half of the sequences from the initial convolution.
For instance, it is sufficient to know the exact formula for the convolution of three among the Fibonacci to Hexanacci and Heptanacci numbers
(so $2\leq m\leq 7$) to know the expression for the convolution of all of the sequences.

In Theorem \ref{th:general_conv_lots_of_seq} in identity \eqref{eq:general_conv_v2_odd} the sequence $F^{(m+2\ell)}_n$ does not take a part in reducing the order of convolution. It is clear that another sequence can be distinguished in a similar way. The change results in different summation or product ranges, which are easily adjustable for a specific example. This implies, for instance, the following.
\begin{corollary}
    We have
    \begin{align*}
        \sum\limits_{K(3,n-2)} F_{k_1}T_{k_2}Q_{k_3} &= \sum\limits_{j=0}^n(T_j-F_j)Q_{n-j}=\sum\limits_{j=0}^{n+1}F_{n+1-j}(Q_j-T_j).
    \end{align*}
\end{corollary}
The above observation also leads to yet another look at the ''switch'' effect described in Section 2. Namely, this effect turns out to be the two different versions of the simplified  convolution sum of three sequences.

Finally, we also note that identities \label{eq:general_conv_v2_even} and \label{eq:general_conv_v2_odd} can be stated in a more general setting. Namely, the left-hand-side of either identity can be any finite product of the terms
$$F^{(m)}_{a}F^{(m+1)}_{b},$$
even with repetitions, and the general formula can be adjusted for that case as well.
We leave the derivation of such a convoluted formula to the reader.

\subsection{Some general cases of convolution of two sequences}

In Section 2, we derived convolution identities that involve all pairs of sequences up to $m=5$.
However, following the proof of Corollary \ref{cor:conv_FQ_FP_TP} we can do more.
In fact, we can deliver the general algorithm for finding a simple and closed form of the convolution of $F^{(m)}$ with $F^{(m+p)}$
in the following cases.\\
\indent Case 1: $p|m$.\\
\indent Case 2: $p|m+1$.\\
\indent Case 3: $p=2m+2$.

We start with Case 1. So let $m=\ell \cdot p$ for some integer $\ell\geq 1$. Applying Theorem \ref{thm:gen_m+p_conv} we get
\begin{equation}\label{eq:Flp-Flp=xlpFF}
    F^{((\ell+1)p)}(x)-F^{(\ell\cdot p)}(x)
    =\left(x^{\ell\cdot p}+\cdots+x^{(\ell+1)p-1}\right)F^{(\ell\cdot p)}(x)F^{((\ell+1)p)}(x).
\end{equation}
Multiplying \eqref{eq:Flp-Flp=xlpFF} repeatedly by $x^p$ we stack up a total of $\ell$ equalities:
\begin{align*}
     x^pF^{((\ell+1)p)}(x)-x^pF^{(\ell\cdot p)}(x)&
     =\left(x^{(\ell+1)p}+\cdots+x^{(\ell+2)p-1}\right)F^{(\ell\cdot p)}(x)F^{((\ell+1)p)}(x),\\
     \vdots & \qquad \vdots\\
     x^{(\ell-1)p}F^{((\ell+1)p)}(x)-x^{(\ell-1)p}F^{(\ell\cdot p)}(x)&
     =\left(x^{(2\ell-1)p}+\cdots+x^{2\ell\cdot p-1}\right)F^{(\ell\cdot p)}(x)F^{((\ell+1)p)}(x).
\end{align*}
Adding everything up we get
\begin{equation}\label{eq:conv_p|m}
     \sum\limits_{j=0}^{\ell-1}x^{j\cdot p}\big(F^{((\ell+1)p)}(x)-F^{(\ell\cdot p)}(x)\big)=\left(x^{\ell\cdot p}+\cdots+x^{2\ell\cdot p-1}\right)F^{(\ell\cdot p)}(x)F^{((\ell+1)p)}(x).
\end{equation}

In the next part of our computation we use the following convention. Whenever we combine several sums into one with fixed summation range, the remaining terms that are not included in the combined sum are called \textit{other terms} and are referred as $o.t.$ Depending in the exact case, these terms can be explicitly derived. We do not do that in the below computation as it makes the formula presented in the algorithm difficult to follow. Thus, only the important terms are explicit.

We go back to \eqref{eq:conv_p|m} and rewrite this into power series to obtain

\begin{align*}\label{eq:conv_p|m__formula}
     \sum\limits_{j=0}^{\ell-1}\big(F^{((\ell+1)p)}_{n-j\cdot p}-F^{(\ell\cdot p)}_{n-j\cdot p}\big)&
     =\sum\limits_{k=0}^{n-(2\ell\cdot p-1)}F^{((\ell+1)p)}_k
     \left(\sum\limits_{j=0}^{\ell\cdot p-1}F^{(\ell\cdot p)}_{n-(2\ell\cdot p-1)-k-j}\right)+o.t.\\
     &=\sum\limits_{k=0}^{n-(2\ell\cdot p-1)}F^{((\ell+1)p)}_kF^{(\ell\cdot p)}_{n-(2\ell\cdot p)-k}+o.t.
\end{align*}
From this we can derive the closed formula for the desired convolution.

\begin{example}
    Set $m=4$ and $p=2$ and denote $s_n=F^{(6)}_n$. It follows that
    \begin{equation}\label{eq:conv_sum_sQ}
        s_{n-2}-Q_{n-2}+s_n-Q_n=\sum\limits_{j=0}^{n-7}s_jQ_{n-3-j}+o.t.
    \end{equation}
    It is now easy to find that
    $$o.t.=2s_{n-6}+s_{n-5},$$
    and this finishes the formula.

    Setting $m=p=2$ restores \eqref{eq:FQ=Q-F}.
\end{example}

We now move to Case 2: $p|m+1$. So, let $m=\ell\cdot p-1$ for some positive integer $\ell$. The further reasoning is similar to case $p|m$. Applying Theorem \ref{thm:gen_m+p_conv} we get
\begin{equation}\label{eq:Flp-Flp=xlpFF_version2}
    F^{((\ell+1)p-1)}(x)-F^{(\ell\cdot p-1)}(x)
    =\left(x^{\ell\cdot p-1}+\cdots+x^{(\ell+1)p-2}\right)F^{(\ell\cdot p-1)}(x)F^{((\ell+1)p-1)}(x).
\end{equation}
We multiply \eqref{eq:Flp-Flp=xlpFF_version2} again by $x^p$ to get
\begin{align*}
     x^pF^{((\ell+1)p-1)}(x)-x^pF^{(\ell\cdot p-1)}(x)&
     =\left(x^{(\ell+1)p-1}+\cdots+x^{(\ell+2)p-2}\right)F^{(\ell\cdot p)}(x)F^{((\ell+1)p)}(x),\\
     \vdots & \qquad \vdots\\
     x^{(\ell-1)p}F^{((\ell+1)p-1)}(x)-x^{(\ell-1)p}F^{(\ell\cdot p-1)}(x)&
     =\left(x^{(2\ell-1)p-1}+\cdots+x^{2\ell\cdot p-2}\right)F^{(\ell\cdot p-1)}(x)F^{((\ell+1)p-1)}(x).
\end{align*}
Adding everything up and passing to power series, we have
\begin{align*}
      \sum\limits_{j=0}^{\ell-1}\big(F^{((\ell+1)p-1)}_{n-j\cdot p}-F^{(\ell\cdot p)-1}_{n-j\cdot p}\big)&
     =\sum\limits_{k=0}^{n-(2\ell\cdot p-2)}F^{((\ell+1)p-1)}_k
     \left(\sum\limits_{j=0}^{\ell\cdot p-1}F^{(\ell\cdot p-1)}_{n-(2\ell\cdot p-2)-k-j}\right)+o.t.\\
     &=2\sum\limits_{k=0}^{n-(2\ell\cdot p-2)}F^{((\ell+1)p)-1}_kF^{(\ell\cdot p-1)}_{n-(2\ell\cdot p-1)-k}+o.t.
\end{align*}

\begin{example}
    Setting $p=3$ and $m=2$ we reproduce the proof of \eqref{FP=P+P=F}. Setting $p=2$ and $m=3$ we reproduce the identity \eqref{PT=P+P-P-P-T-T+P+P}.
\end{example}

Finally, consider Case 3 and $p=2m+2$. The key feature of this case is the following simple lemma.
\begin{lemma}
    For any $m\geq 2$ and any $n\geq 0$ we have
    \begin{equation}\label{eq:sum_F_n^m=4F_n+2m}
    \sum\limits_{k=0}^{2m+1}F_{n+k}^{(m)}=4F_{n+2m}^{(m)}.
    \end{equation}
\end{lemma}
\begin{proof}
    Write
    \begin{align*}
        4F_{n+2m}^{(m)}&=\big(2F_{n+m}^{(m)}+\cdots+2F_{n+2m-1}^{(m)}\big)+2F_{n+2m}^{(m)}\\
        &=2F_{n+m}^{(m)}+\big(2F_{n+m+1}^{(m)}+\cdots+2F_{n+2m}^{(m)}\big)\\
        &=F_m^{(m)}+\cdots+F_{n+m-1}^{(m)}+F_{n+m}^{(m)}\\
        &\qquad +F_{n+m+1}^{(m)}+\cdots+F_{n+2m}^{(m)}+F_{n+2m+1}^{(m)}.
    \end{align*}
\end{proof}
We now apply Theorem \ref{thm:gen_m+p_conv} to obtain
$$F^{3m+2}(x)-F^{(m)}(x)=\big(x^m+\cdots+x^{3m+1}\big) F^{(m)}(x)F^{(3m+2)}(x).$$
This implies, using \eqref{eq:sum_F_n^m=4F_n+2m}, that
\begin{align*}
    F_{n}^{(3m+2)}-F_n^{(m)}&=\sum\limits_{k=0}^{n-(3m+1)}F_{k}^{(3m+2)}\big(F_{n-(3m-1)-k}^{(m)}+\cdots+F_{n-m-k}^{(m)}\big)+o.t.\\
    &=4\sum\limits_{k=0}^{n-(3m+1)}F_{k}^{(3m+2)}F_{n-m-1-k}^{(m)}+o.t.
\end{align*}
\begin{example}
    Set $m=2$ and let $\mathcal{O}=F^{(8)}$. Then we have
    \begin{align*}
        \mathcal{O}_{n}-F_n&=4\sum\limits_{j=0}^{n-7}
\mathcal{O}_jF_{n-3-j}+\mathcal{O}_{n-3}+2\mathcal{O}_{n-4}+4\mathcal{O}_{n-5}+7\mathcal{O}_{n-6}\\
&=4\sum\limits_{j=0}^{n-3}\mathcal{O}_jF_{n-3-j}+\mathcal{O}_{n-3}-2\mathcal{O}_{n-4}-\mathcal{O}_{n-6}.
    \end{align*}
    It follows after some calculation that
    \begin{equation}\label{eq:conv_F&Oct}
        \sum\limits_{j=0}^{n}\mathcal{O}_kF_{n-j}=\frac{1}{4}\left(\mathcal{O}_{n+3}-\mathcal{O}_{n}+2\mathcal{O}_{n-1}+\mathcal{O}_{n-3}-F_{n+3}\right).
    \end{equation}
\end{example}

The methods provided to this point allow us to find a closed form of the convolution of two different sequences out of the set of Fibonacci, Tribonacci, Tetranacci and Pentanacci numbers. If we include Hexanacci numbers (i.e. $s=F^{(6)}$), then we can compute all convolutions but the convolution of Fibonacci and Hexanacci numbers. This is at first glance surprising, but in fact the presented algorithm does not allow us to deal with that case, even though we can compute the explicit form of $\sum_{j=0}^n F^{(20)}_j F^{(27)}_{n-j}$.

We now show how to find the sum $\sum_{j=0}^ns_jF_{n-j}$. In order to deal with this problem, we have to use another approach.

\begin{theorem}
    For $n\geq 3$ we have
    \begin{equation}\label{eq:conv_sF}
        \sum\limits_{j=0}^{n}s_jF_{n-j}=\frac{1}{5}\left(s_{n+3}+s_{n+1}-s_n+3s_{n-1}+s_{n-3}-F_{n+3}-F_{n+1}\right).
    \end{equation}
\end{theorem}
\begin{proof}
    Notice that by \eqref{Eq:Fab-Fa=xFaFab} or \eqref{eq:sumF^mF^m+p=F^m+p-F^m} we have
\begin{align*}
    s_n-F_n&=\sum\limits_{j=0}^{n-2}s_jF_{n-j-2}+\sum\limits_{j=0}^{n-3}s_jF_{n-j-3}+\sum\limits_{j=0}^{n-4}s_jF_{n-j-4}+\sum\limits_{j=0}^{n-5}s_jF_{n-j-5},\\
    s_{n+2}-F_{n+2}&=\sum\limits_{j=0}^{n}s_jF_{n-j}+\sum\limits_{j=0}^{n-1}s_jF_{n-j-1}+\sum\limits_{j=0}^{n-2}s_jF_{n-j-2}+\sum\limits_{j=0}^{n-3}s_jF_{n-j-3}.
\end{align*}
Summing up and rearranging we get
\begin{align*}
    s_{n+2}-F_{n+2}+s_n-F_n&=\sum\limits_{j=0}^{n-5}s_j(F_{n-j-5}+F_{n-j-4}+2F_{n-j-3}+2F_{n-j-2}+F_{n-j-1}+F_{n-j})\\
    &\quad\quad +s_{n-1}+2s_{n-2}+5s_{n-3}+9s_{n-4}.
\end{align*}
To proceed further, we apply the identity
\begin{equation}\label{eq:F_j+...+F_j+5=5Fj+4}
    F_n+F_{n+1}+2F_{n+2}+2F_{n+3}+F_{n+4}+F_{n+5}=5F_{n+4}
\end{equation}
valid for any $n\geq 0$ and we substitute $n\to n+1$ to obtain
\begin{align*}
    s_{n+3}-F_{n+3}+s_{n+1}-F_{n+1}&=\sum\limits_{j=0}^{n-5}5s_jF_{n-j-1}+s_{n-1}+2s_{n-2}+5s_{n-3}+9s_{n-4}\\
    &=5\sum\limits_{j=0}^ns_jF_{n-j}+s_n-3s_{n-1}-s_{n-3}.
\end{align*}
Thus, after minor adjustments, we have \eqref{eq:conv_sF}.
\end{proof}

We note that there are more cases where an identity similar to \eqref{eq:F_j+...+F_j+5=5Fj+4} leads to a closed sum formula.
Namely, if we consider $p+m\leq 8$ and let $S=F^{(7)}$ (Heptanacci numbers), then the only missing cases, not following from
the rules described by the three cases, are
$$
\sum\limits_{j=0}^{n} S_j F_{n-j},\qquad \sum\limits_{j=0}^{n}S_j Q_{n-j} \qquad\text{and}\qquad
\sum\limits_{j=0}^n \mathcal{O}_jT_{n-j}.
$$
These sums can be derived using an approach similar to \eqref{eq:conv_sF}, but this time with the aid of the following identities:
\begin{align}
\label{eq:weird_sums_1}    2 F_n + 2 F_{n+1} + 3 F_{n+2} + 3 F_{n+3} + 3 F_{n+4} + F_{n+5} + F_{ n+6} &= 11 F_{n+4},\\
    \label{eq:weird_sums_2}Q_n+Q_{n+1}+Q_{n+2}+Q_{n+3}+Q_{n+4}+2Q_{n+5}+2Q_{n+6}+2Q_{n+7}+Q_{n+8}&=3Q_{n+8},\\
    \label{eq:weird_sums_3}2T_n+2T_{n+1}+3T_{n+2}+5T_{n+3}+5T_{n+4}+3T_{n+5}+3T_{n+6}+2T_{n+7}&=11T_{n+6}.
\end{align}
To clarify how to use them, we write, for example,
\begin{align*}
    &Q_n+Q_{n+1}+Q_{n+2}+Q_{n+3}+Q_{n+4}+2Q_{n+5}+2Q_{n+6}+2Q_{n+7}+Q_{n+8}\\
    &\qquad =(Q_n+Q_{n+1}+Q_{n+2})+(Q_{n+3}+Q_{n+4}+Q_{n+5})\\
    &\qquad \qquad +(Q_{n+5}+Q_{n+6}+Q_{n+7})+(Q_{n+6}+Q_{n+7}+Q_{n+8})
\end{align*}
and it is clear that each bracket can be generated from the identity
$$S(x)-Q(x)=(x^4+x^5+x^6)Q(x)S(x).$$

\subsection{Open problems}

In the previous section we presented the algorithm for computation of the convolution sum of two sequences under (major) restrictions. We dealt with missing case $m=2$ and $p=4$ separately so that all convolution sums with $m+p\leq 8$ for $m\geq 2$ and $p\geq 1$ have their closed forms calculated. The first case that is not covered by our methods (that is, the case with smallest possible $m+p$ and smallest possible $p$) is the following convolution sum (see also Table \ref{tab:check}):
$$
\sum\limits_{j=0}^n P_j F^{(9)}_{n-j}.
$$
The trick that was used above could also work here but this does not replace a general approach to these sums (identities \eqref{eq:F_j+...+F_j+5=5Fj+4}--\eqref{eq:weird_sums_3} seem to only work in the presented form, we do not know if/how they generalize, as they were found by trial and error).
In our opinion, a good starting point is to search for an identity of the form
\begin{equation}\label{eq:general_identity_sums_arithmetic}
    \sum\limits_{k\in K}\sum\limits_{j=0}^{p-1}F_{n+j+k}^{(m)} = N\cdot F_{n+\ell}^{(m)}
\end{equation}
valid for any $n$, $m\geq 2$, with $N$ and $\ell$ being unknown, $K$ being a finite set, $\ell$ related to $n$ and $p$.
Identity \eqref{eq:F_j+...+F_j+5=5Fj+4} is the case $m=2$ and follows that pattern with $K=\{0,2\}$, $p=4$, $N=5$ and $\ell=4$. The identity \eqref{eq:sum_F_n^m=4F_n+2m} is the simplest example of that form, with $K=\{0\}$.
Any identity of the form \eqref{eq:general_identity_sums_arithmetic} would give us yet another convolution sum.
We believe that finding any other solution (or even an infinite family of solutions) to that equation is a good
motivation for further research in the topic.

\begin{small}
\begin{table}[h]
\begin{tabular}{@{}cccccccccc@{}}
\toprule
$m\ \backslash\ p$ & 2     & 3     & 4     & 5     & 6      & 7     & 8      & 9     \\ \midrule
2   & \eqref{FQ=Q+Q-F}    & \eqref{FP=P+P=F}  & \eqref{eq:conv_sF}    & \eqref{eq:weird_sums_1}      & \eqref{eq:conv_F&Oct} & ?    & ?     & ?    \\
3   & \eqref{PT=P+P-P-P-T-T+P+P} & $p|m$   & $p|m+1$ & \eqref{eq:weird_sums_3}    & ?     & ?    & $p=2m+2$ & ?    \\
4   & \eqref{eq:conv_sum_sQ}   & \eqref{eq:weird_sums_2}    & $p|m$   & $p|m+1$ & ?     & ?    & ?     & ?    \\
5   & $p|m+1$ & $p|m+1$ & ?    & $p|m$  & $p|m+1$  & ?    & ?     & ?    \\
6   & $p|m$  & $p|m$  & ?    & ?    & $p|m$   & $p|m+1$ & ?     & ?    \\
7   & $p|m+1$ & ?    & $p|m+1$ & ?    & ?     & $p|m$  & $p|m+1$  & ?    \\
8   & $p|m$  & $p|m+1$ & $p|m$  & ?    & ?     & ?    & $p|m$   & $p|m+1$ \\
9   & $p|m+1$ & $p|m$  & ?    & $p|m+1$ & ?     & ?    & ?     & $p|m$ \\
\bottomrule
\end{tabular}
\caption{Convolution sums $\sum F^{(m)}_jF^{(m+p)}_{n-j}$ with $2\leq m,p\leq 9$ covered directly or indirectly in this article. The cases $m=1$ and $p=1$ are covered by \eqref{eq:partial_sum_formuka} and \eqref{eq:simplest_conv_sum_in_the_article}, respectively. Question marks indicate unsolved cases.}
\label{tab:check}
\end{table}
\end{small}

\section{Conclusion}

This article was devoted to study convolutions involving Fibonacci $m$-step numbers.
We have applied the prominent generating function approach to prove several appealing results
that strengthen the understanding of these numbers. Many known identities for Fibonacci, Tribonacci,
Tetranacci and Pentanacci numbers now follow from our results as special cases. In addition,
mixed convolutions involving Fibonacci $m$-step numbers and Jacobsthal and Pell numbers were stated.
To keep things coherent and streamlined, we focused exclusively on Fibonacci $m$-step numbers.
There is still much work to be done. Identities for Lucas $m$-step numbers, Pell $m$-step numbers, and others,
and also mixed convolutions of these sequences can be studied in the future.

\makeatletter
% \section*{Acknowledges}
% I thank ...

\end{document}